\documentclass[11pt, reqno]{amsart}
\usepackage[top=4cm, bottom=4cm, left=3cm, right=3cm]{geometry}
\usepackage{amsthm}
\usepackage{amsmath}
\usepackage{amssymb}
\usepackage{mathtools}
\usepackage{enumitem}

\usepackage{bm}
\usepackage{ytableau}
\usepackage{tikz}

\newcommand{\N}{\mathbb{N}}

\newcommand{\Z}{\mathbb{Z}}

\newcommand{\C}{\mathbb{C}}

\newcommand{\SSYT}{{\rm SSYT}}
\newcommand{\Rect}{{\rm Rect}}
\newcommand{\Sym}{{\rm Sym}}
\newcommand{\shape}{{\rm shape}}

\newtheorem{thms}{Theorem}[section]
\newtheorem{lem}[thms]{Lemma}
\newtheorem{cor}[thms]{Corollary}
\newtheorem{prop}[thms]{Proposition}

\theoremstyle{definition}
\newtheorem{dfn}[thms]{Definition}
\newtheorem{ex}[thms]{Example}
\newtheorem{rmk}[thms]{Remark}

\author{Hikari Hanaki}
\address{Mathematical Institute, Tohoku University, 6-3 Aramaki Aza-Aoba, Aoba-ku, Sendai, 980-8578, Japan.}

\email{hanaki.hikari.s8@dc.tohoku.ac.jp}
\thanks{}


\date{}

\dedicatory{}


\begin{document}

\allowdisplaybreaks

\title{The Littlewood-Richardson rule for Schur multiple zeta functions}

\maketitle

\begin{abstract}
The Schur multiple zeta function was defined as a multivariable function by Nakasuji-Phuksuwan-Yamasaki. 
Inspired by the product formula of Schur functions, the products of Schur multiple zeta functions have been studied, for example, by Nakasuji--Takeda and by Nakaoka. 
While a product of two Schur functions expands as a linear combination of Schur functions, it is known that  a similar expansion for the product of Schur multiple zeta functions can be obtained by symmetrizing, i.e., by taking the summation over all permutations of the variables. In this paper, we present a more refined formula by restricting the summation from the full symmetric group to its specific subgroup. 
\end{abstract}

\section{Introduction}

The {\it Schur multiple zeta functions} were introduced by Nakasuji-Phuksuwan-Yamasaki \cite{NPY} as a generalization of both multiple zeta and multiple zeta-star functions of Euler-Zagier type.
These are defined as sums over combinatorial objects called semi-standard Young tableaux, similar to usual Schur functions.
We now review the detailed definition. 

A {\it partition} is a finite sequence 
$\lambda=(\lambda_1, \dots, \lambda_r)$ of non-negative integers in weakly decreasing order: $\lambda_1 \ge \dots \ge \lambda_r$. The number of non-zero $\lambda_i$ in $\lambda=(\lambda_1, \dots, \lambda_r)$ is the {\it length} of $\lambda$, denoted by $l(\lambda)$, and the sum of $\lambda_i$ is the {\it weight} of $\lambda$, denoted by $|\lambda|$. If $|\lambda|=n$, we say that $\lambda$ is a partition of $n$. 
Throughout this paper, let $\lambda, \mu, \nu$ be partitions.

For $\lambda = (\lambda_1, \dots, \lambda_r)$, let $D(\lambda)$ be the subset of $\Z^2$:
$D(\lambda)=\{(i,j)\in\Z^2|1\le i\le r, 1\le j\le \lambda_i\}$
and depicted as a collection of square boxes arranged in left-justified rows with $\lambda_i$ boxes in the $i$-th row. We call this the {\it Young diagram} (or simply diagram) of shape $\lambda$. 
We identify $\lambda$ with this collection of boxes, and we call it {\it shape} $\lambda$ and denote it by the same symbol $\lambda$. 
We say that $(i,j)\in D(\lambda)$ is a {\it corner} of $\lambda$ if $(i+1,j)\notin D(\lambda)$ and $(i,j+1)\notin D(\lambda)$ and we denote by $C(\lambda)\subset D(\lambda)$ the set of all corners of $\lambda$. The {\it conjugate} of a partition $\lambda$ is the partition $\lambda'=(\lambda'_1, \lambda'_2, \dots)$ defined by $\lambda'_j=\#\{i\>|\>\lambda_i\ge j\}$.  

A {\it skew partition} $\lambda/\mu$ is a pair of partitions $\lambda$ and $\mu$ such that $\mu_i \le \lambda_i$ for all $i$. The Young diagram of a skew partition $\lambda/\mu$, denoted by $D(\lambda/\mu)$, is defined by $D(\lambda/\mu) = D(\lambda) \setminus D(\mu)$ and is depicted as the diagram obtained by removing that of $\mu$ from that of $\lambda$. For clarity, sometimes we depict the diagram of $\lambda/\mu$ by coloring the boxes of $\mu$ in black as in Figure \ref{skew_diag}. 
When $\mu = \emptyset$ (i.e., $\mu_i = 0$ for all $i$), we identify skew partition $\lambda/\mu$ with $\lambda$ and call it {\it normal partition}. Unless otherwise stated, partitions are assumed to be non-empty. The corners of $\lambda/\mu$ are defined the same as the case of normal partition.

\begin{figure}
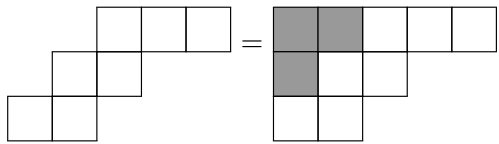

\begin{ytableau}
\none & \none & {} & {} & {}\\
\none & {} & {} \\
{} & {}
\end{ytableau}
=
\begin{ytableau}
*(black!40){} & *(black!40){} & {} & {} & {}\\
*(black!40){} & {} & {} \\
{} & {}
\end{ytableau}
\caption{Example of diagram of skew partition : $\lambda/\mu = (5, 3, 2)/(2, 1)$}
\label{skew_diag}
\end{figure}

Let $X$ be a set. A {\it Young tableau} (or simply tableau) $T=(t_{ij})$ of shape $\lambda$ over $X$ is a filling of $D(\lambda)$ obtained by putting $t_{ij}\in X$ into the $(i,j)$ box of $D(\lambda)$. We denote by $T(\lambda, X)$ the set of all Young tableaux of shape $\lambda$ over $X$. 
We denote the shape of a Young tableau $T$ by $\shape(T)$.

A {\it semi-standard Young tableau} is a Young tableau over the set of positive integers $\N$ such that the entries in each row are weakly increasing from left to right and those in each column are strictly increasing from top to bottom. 
We denote by $\SSYT(\lambda)$ the set of all semi-standard Young tableaux of shape $\lambda$. 
A skew Young tableau and a skew semi-standard Young tableau are defined similarly.
We may denote by $\SSYT$ the set of all (normal or skew) semi-standard Young tableaux for short.

Let $\delta = \lambda/\mu$ (now, including $\mu = \emptyset$).  For a given tableau of variables $\bm{s}=(s_{ij})\in T(\delta, \mathbb{C})$, 
{\it Schur multiple zeta function} of shape $\delta$ is defined as
$$
\zeta_{\delta}(\bm{s})=\sum_{M\in \SSYT(\delta)}\frac{1}{M^{\bm{s}}},
$$
where $M^{\bm{s}}=\displaystyle{\prod_{(i, j)\in D(\delta)}m_{ij}^{s_{ij}} }$ for $M=(m_{ij})\in \SSYT(\delta)$. 
The function $\zeta_\delta(\bm{s})$ converges absolutely in 
$$
 W_{\delta}
=
\left\{\bm{s}=(s_{ij})\in T(\delta,\mathbb{C})\,\left|\,
\begin{array}{l}
 \text{$\Re(s_{ij})\ge 1$ for all $(i,j)\in D(\delta) \setminus C(\delta)$}, \\[3pt]
 \text{$\Re(s_{ij})>1$ for all $(i,j)\in C(\delta)$}
\end{array}
\right.
\right\}.
$$

{\it Schur function} for a normal or skew partition $\delta$ is a symmetric function such that 
$$s_\delta = \sum_{(m_{ij}) \in \SSYT(\delta)} \prod_{(i, j)\in D(\delta)} x_{m_{ij}}.$$
Since the Schur multiple zeta functions have a structure similar to that of the Schur functions, they are expected to have similar properties to those for Schur functions. 
It is known that Schur functions for normal partitions form a $\Z$-basis of the ring of symmetric functions (which is denoted by $\Lambda$ in \cite{M}), and then the {\it Littlewood-Richardson coefficients} $c_{\mu\nu}^\lambda$ are defined by the expansion
\begin{align}\label{LR1}
s_\mu s_\nu=\sum_{\lambda: \rm{partition}}c_{\mu\nu}^\lambda s_\lambda 
\end{align}
for partitions $\lambda, \mu, \nu$. 
Also it holds that 
\begin{align}\label{LR2}
s_{\lambda/\mu}=\sum_{\nu:\text{partition}}c_{\mu\nu}^\lambda s_\nu
\end{align}
for skew Schur functions. 
For the product of Schur multiple zeta functions $\zeta_{\mu}(\bm{s})$ and $\zeta_{\nu}(\bm{t})$, Nakaoka \cite{N} showed that an analogue of (1) can be obtained by taking the summation over $\Sym(\bm{s} \ast \bm{t})$, which is the symmetric group permuting all the  variables $\{s_{ij}\} \cup \{t_{ij}\}$. 

In this paper, we obtain more refined formula by restricting the $\Sym(\bm{s} \ast \bm{t})$ to its subgroup $\Sym(B(\bm{s} \ast \bm{t}))$. 
To restrict the symmetric group, we use a combinatorial rule for computing the Littlewood-Richardson coefficients.

\begin{thms}\label{main}
Let $\mu, \nu$ be partitions, and let $\mu', \nu'$ be the conjugates of $\mu, \nu$, respectively. 
Let $\bm{s}=(s_{ij})\in T(\mu, \mathbb{C}), \bm{t}=(t_{ij})\in T(\nu,\mathbb{C})$, and assume that the entries are variables. 
Assume that the real parts of the variables
$s_{i1}$ with $\mu'_2+\nu'_1\le i\le\mu'_1 - 1$and $t_{1j}$ with $\nu_2+\mu_1\le j\le \nu_1 - 1$
are greater than or equal to 1 and the real parts of all other variables are greater than 1.
Then the following equality holds: 
for any $(\bm{u}_{\lambda}(\bm{s} \ast \bm{t})\in U_\lambda(\bm{s} \ast \bm{t}))_{\lambda\in\mathcal{G}(\mu \ast \nu)}$,

\begin{align*}
\sum_{\Sym(B(\bm{s} \ast \bm{t}))}\zeta_{\mu}(\bm{s})\zeta_{\nu}(\bm{t})=\sum_{\Sym(B(\bm{s} \ast \bm{t}))}\sum_{\lambda\in\mathcal{G}(\mu \ast \nu)}c_{\mu\nu}^{\lambda}\zeta_{\lambda}(\bm{u}_{\lambda}(\bm{s} \ast \bm{t})).
\end{align*}
Here $\sum_{\Sym(B(\bm{s} \ast \bm{t}))}$ is the symmetric group permuting all the variables except  
$\{s_{i1}\>|\>\mu'_2+\nu'_1\le i\le\mu'_1\}\sqcup\{t_{1j}\>|\>\nu_2+\mu_1\le j\le \nu_1\}.$
\end{thms}
Where $\mathcal{G}(\mu \ast \nu)$ is the set of all partitions $\lambda$ such that $c_{\mu\nu}^\lambda > 0$, but we define it in another way in \S3. The more precise definition of the symmetric group $\Sym(B(\bm{s} \ast \bm{t}))$ and the set of tableau $U_\lambda(\bm{s} \ast \bm{t})$ will be provided in \S \ref{prep} (Definition \ref{arm_body_def}) and \S \ref{proof_skew_subs} respectively. 

Furthermore, the following result can be obtained as an analogy of equality (2), and Theorem \ref{main} can be seen as a specialization of Theorem \ref{skew}.

\begin{thms}\label{skew}
Let $\lambda/\mu$ be a skew partition, and let $\lambda', \mu'$ be the conjugates of $\lambda, \mu$, respectively. Let $\bm{v} = (v_{ij}) \in T(\lambda/\mu, \C)$, and assume that the entries are variables. 
Assume that the real parts of the variables
$v_{1j}$ with $\displaystyle{{\rm min}\{\lambda_2, \sum_{i \ge 2}(\lambda_i - \mu_i)\}} + \mu_1 \le j \le \lambda_1 - 1$ and  
$v_{i1}$ with $\displaystyle{{\rm min} \{\lambda'_2, \sum_{j \ge 2}(\lambda'_j - \mu'_j)\}} + \mu'_1 \le i \le \lambda'_1 - 1$ 
are greater than or equal to 1 and the real parts of all other variables are greater than 1.
Then the following equality holds: for any $(\bm{u}_\nu(\bm{v}) \in U_{\nu}(\bm{v}))_{\nu \in \mathcal{G}(\lambda/\mu)}$, 
\begin{align*}
\sum_{\Sym(B(\bm{v}))} \zeta_{\lambda/\mu}(\bm{v})=\sum_{\Sym(B(\bm{v}))} \sum_{\nu \in \mathcal{G}(\lambda/\mu)} c_{\mu\nu}^{\lambda}\zeta_{\nu}(\bm{u}_{\nu}(\bm{v})).
\end{align*}
Here $\sum_{\Sym(B(\bm{v}))}$ denotes the summation over the permutations of all variables in $B(\bm{v})$.
\end{thms}

The set $\mathcal{G}(\lambda/\mu)$ is the set of all partitions $\lambda$ such that $c_{\mu\nu}^\lambda > 0$.
The tableau $B(\bm{v})$ which is called the {\it body} of $\bm{v}$, and the set of tableau $U_\nu(\bm{v})$ are defined in \S3.2 more precisely.

\begin{rmk}\label{body_rmk}
The body of $\bm{v}$ is obtained by removing contiguous boxes from the bottom to up in the first column, and removing contiguous these from the rightmost to the left in the first row, from $\bm{v}$. (See Figure \ref{fig_body}.)

\begin{figure}[ht]
$$
\vcenter{\hbox{
\begin{tikzpicture}[baseline=(current bounding box.center)]
\node at (0,0){ 
~\begin{ytableau}
\none & \none & *(black!10)v_{13} & *(black!10)v_{14} & v_{15} & v_{16} \\
\none & *(black!10)v_{22} & *(black!10)v_{23}  \\
*(black!10)v_{31} & *(black!10)v_{32} & *(black!10)v_{33} \\
*(black!10)v_{41} \\
v_{51}
\end{ytableau}~};
\draw[line width=0.7pt] (0.2, -0.2) -- (0.8, -0.8) node[at={(1.4, -0.8)}]{$B(\bm{v})$};
\end{tikzpicture}. }}
$$
\caption{Example of body $\bm{v}$}
\label{fig_body}
\end{figure}

\end{rmk}

\begin{rmk}
For $\bm{s} \in T(\mu, \C)$ and $\bm{t} \in T(\nu, \C)$, $\bm{s} \ast \bm{t}$ is defined by putting $\bm{s}$ below and $\bm{t}$ to the right of the rectangle of empty boxes with $\mu_1$ columns and $\nu'_1$ rows. (See Figure \ref{fig_ast}.)

\begin{figure}[ht]
$$
\bm{s} \ast \bm{t} =
\vcenter{\hbox{
\begin{tikzpicture}[baseline=(current bounding box.center)]
\node at (0,0){ 
~\begin{ytableau}
\none & \none & \none & {} & {} & {} & {} \\
\none & \none & \none & {} & {} \\
*(black!10){} & *(black!10){} & *(black!10){} \\
*(black!10){} & *(black!10){} \\
*(black!10){} 
\end{ytableau}~};
\node at (-1.4, -0.3)[fill = black!10]{$\ \bm{s}\ $};
\node at (0.4, 0.85) [shape=circle] [fill = white!50]{$\bm{t}$};
\end{tikzpicture}. }}$$
\caption{Example of $\bm{s} \ast \bm{t}$}
\label{fig_ast}
\end{figure}

We define $\bm{s} \ast \bm{t}$ in \S3.3 again. 
Let $\delta = \shape(\bm{s} \ast \bm{t})$. 
By discussion in \S3.3, $\zeta_{\delta}(\bm{s} \ast \bm{t}) = \zeta_{\mu}(\bm{s}) \zeta_{\nu}(\bm{t})$. Theorem \ref{main} follows from applying Theorem $\ref{skew}$ to the skew tableau $\bm{s} \ast \bm{t}$ with some further discussions. 
\end{rmk}

This paper is organized as follows. 
In \S2, we introduce Knuth equivalence and jeu de taquin, which are the combinatorial background of semi-standard Young tableau. In \S3, we give proofs of Theorem \ref{main} and Theorem \ref{skew}. Finally, in \S4, we discuss further generalization of Theorem \ref{skew}.

\section{Combinatorial background}
This section is devoted to a review of Knuth equivalence and jeu de taquin, and the Littlewood-Richardson rule. 
These will be used to restrict the group action in Theorem \ref{skew} in the proof.

\subsection{Knuth equivalence}\label{Knuth}
We write words as a sequence of letters 
and write $w \cdot w'$ or $ww'$ for the word which is the juxtaposition of the two words $w$ and $w'$.
The {\it word} or {\it row-word} of $L \in \SSYT$ is defined by reading the entries of $L$ ``from left to right and bottom to top", i.e., starting with the bottom row, writing down its entries from left to right, then listing the entries from left to right in the next to the
bottom row and working up to the top. 
The row-word of $L$ is denoted by $w_{row}(L)$. 

\begin{dfn}[Knuth relation]
Two words $w, w'$ over $\N$ are related by \it{Knuth relation} if they satisfy one of the following conditions.

\begin{enumerate}
\item[(K)]
The words $w$ and $w'$ are of the form $w = u \cdot bca \cdot v$ and $w' = u \cdot bac \cdot v$ (or vice-versa) for some $a, b, c \in \N$ with $a < b \le c$ and some words $u, v$ over $\N$ (but $u, v$ can be the empty word). 

\item[(K')]
The words $w$ and $w'$ are of the form $w = u \cdot acb \cdot v$ and $w' = u \cdot cab \cdot v$ (or vice-versa) for some $a, b, c \in \N$ with $a \le b < c$ and some words $u, v$ over $\N$ (but $u, v$ can be the empty word). 
\end{enumerate}
\end{dfn}

For example, 
$231$ and $213$ are related by (K). 
Also, $15\cdot 132 \cdot 4$ and $15 \cdot 312 \cdot 4$ are related by (K'). 
For two words $w, w'$, we say that $w$ is {\it Knuth equivalent} to $w'$ if they are related by the transitive closure of the Knuth relations and we write $w \equiv w'$. 
The following theorem is known about Knuth equivalence.

\begin{thms}[cf. {\cite[\S2]{F}}]\label{inv}
Every word $w$ over $\N$ is Knuth equivalent to the row-word of a unique semi-standard Young tableau of normal shape.
\end{thms}

For a word $w$, we denote such a unique normal tableau by $P(w)$, i.e., $w \equiv w_{row}(P(w))$.

\subsection{Jeu de Taquin}
In this subsection, we review the operation {\it jeu de taquin}, which was introduced by Sch\"{u}tzenberger. It is an operation that transforms a skew tableau into a normal tableau. 

First, we define the operation called {\it elementary slide}. An elementary slide is performed on a tableau $L$ over $\N$ such that one entry is not in $\N$ but is a dot. Thus, if $a$ is the element below the dot and $b$ is the element to the right of the dot, then we replace the dot with $a$ if $a \le b$ and with $b$ if $b < a$. If either $a$ or $b$ is missing from the tableau, we replace the dot with the present entry. Elementary slides are usually as shown in Figure \ref{slide}.

A skew diagram $\lambda/\mu$ has one or more {\it inside corners}. An inside corner is a corner of the smaller (deleted) diagram $\mu$. An {\it outside corner} is a corner of $\lambda$. 
The operation of a {\it slide} takes a tableau $L = (l_{ij}) \in \SSYT(\lambda/\mu)$ and an inside corner $c$. Assume that a dot is put into the inside corner $c$.

The slide is completed by repeating elementary slides until the dot is placed in an outside corner.
If the dot is placed in an outside corner of the tableau, we remove both the box containing the dot and the dot itself.

\begin{figure}[h]
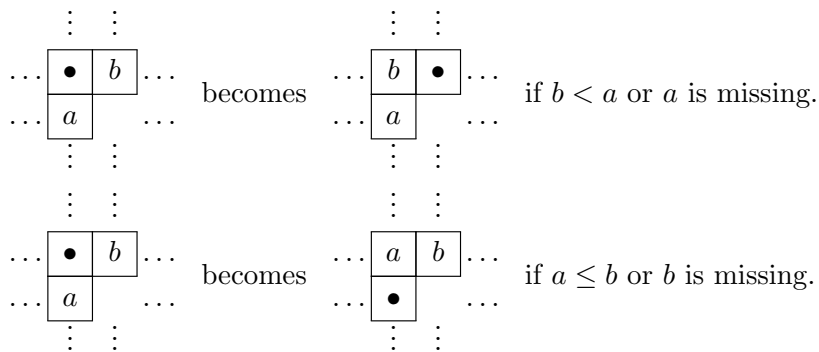

\begin{itemize}[label = {}]
\item
$
\vcenter{\hbox{
~\begin{ytableau}
\none & \none[\vdots] & \none[\vdots] \\
\none[\dots] & \bullet & b & \none[\dots] \\
\none[\dots] & a & \none & \none[\dots]\\
\none & \none[\vdots] & \none[\vdots]
\end{ytableau}~}}$
becomes 
$
\vcenter{\hbox{
~\begin{ytableau}
\none & \none[\vdots] & \none[\vdots] \\
\none[\dots] & b & \bullet & \none[\dots] \\
\none[\dots] & a & \none & \none[\dots]\\
\none & \none[\vdots] & \none[\vdots]
\end{ytableau}~}}$
if $b < a$ or $a$ is missing.
\item
$
\vcenter{\hbox{
~\begin{ytableau}
\none & \none[\vdots] & \none[\vdots] \\
\none[\dots] & \bullet & b & \none[\dots] \\
\none[\dots] & a & \none & \none[\dots]\\
\none & \none[\vdots] & \none[\vdots]
\end{ytableau}~}}$
becomes 
$
\vcenter{\hbox{
~\begin{ytableau}
\none & \none[\vdots] & \none[\vdots] \\
\none[\dots] & a & b & \none[\dots] \\
\none[\dots] & \bullet & \none & \none[\dots]\\
\none & \none[\vdots] & \none[\vdots]
\end{ytableau}~}}$
if $a \le b$ or $b$ is missing.
\end{itemize}
\caption{Elementary slides}
\label{slide}
\end{figure}

\begin{figure}[h]
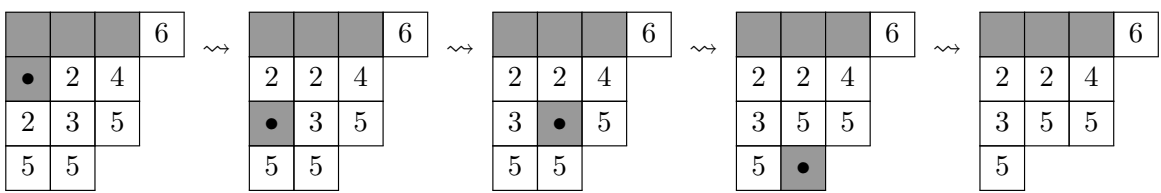

$$~\begin{ytableau}
*(black!40){} & *(black!40){} & *(black!40){} & 6 \\
*(black!40) \bullet & 2 & 4 \\
2 & 3 & 5\\
5 & 5
\end{ytableau}~\leadsto
~\begin{ytableau}
*(black!40){} & *(black!40){} & *(black!40){} & 6 \\
2 & 2 & 4 \\
*(black!40)\bullet & 3 & 5\\
5 & 5
\end{ytableau}~\leadsto
~\begin{ytableau}
*(black!40){} & *(black!40){} & *(black!40){} & 6 \\
2 & 2 & 4 \\
3 & *(black!40)\bullet & 5\\
5 & 5
\end{ytableau}~\leadsto
~\begin{ytableau}
*(black!40){} & *(black!40){} & *(black!40){} & 6 \\
2 & 2 & 4 \\
3 & 5 & 5\\
5 &*(black!40)\bullet
\end{ytableau}~\leadsto
~\begin{ytableau}
*(black!40){} & *(black!40){} & *(black!40){} & 6 \\
2 & 2 & 4 \\
3 & 5 & 5\\
5 
\end{ytableau}~$$
\caption{Example of slide}
\end{figure}

Jeu de taquin is performed as a succession of {\it slides} on a tableau as follows. 
Given a tableau $L \in \SSYT(\lambda/\mu)$, the process of slide can be carried out from any inside corner. Another inside corner can be chosen for the resulting tableau, and the procedure repeated, until there are no more inside corners. This means that the resulting tableau has a normal shape. 
The whole process is called the {\it jeu de taquin}. 
The following theorem is known.

\begin{thms}[cf. {\cite[\S2]{F}}]\label{rect_word}
Suppose a slide can be performed on $L \in \SSYT$ to produce $M \in \SSYT$. Then the row-word  
of $L$ is Knuth equivalent to the row-word of $M$, i.e., 
$$w_{row}(L) \equiv w_{row}(M).$$
\end{thms}

Combining Theorem \ref{inv} and Theorem \ref{rect_word}, we have that $L \in \SSYT$ becomes the normal tableau $P(w_{row}(L))$ by jeu de taquin. We define ${\rm Rect}(L)$ as $P(w_{row}(L))$ for $L \in \SSYT$.

The following theorem about jeu de taquin is one of the combinatorial rules for computing the Littlewood-Richardson coefficients $c_{\mu\nu}^\lambda$.

\begin{thms}[cf. {\cite[\S5]{F}}]\label{LRrule2}
For any $M \in \SSYT(\nu)$, 
$$
\#\{L \in \SSYT(\lambda/\mu) \ |\ \Rect(L) = M\} = c_{\mu\nu}^\lambda.
$$
\end{thms}

\section{Proof of main Theorems}

In this section, we prove the main theorems. 
In \S\ref{prep}, we introduce $\tilde{\N}$, which is the set of positive integers labeled by positive integers, in order to distinguish each entry in semi-standard Young tableau. 
In \S3.2, we prove Theorem \ref{skew}, and in \S3.3, we prove the Theorem \ref{main} by using Theorem \ref{skew}. 

\subsection{Preparation}\label{prep}

Let $\tilde{\N}=\{k_l\ |\ k, l \in \N\}$ be a set of elements endowed with the total order
$$1_1 < 1_2 < \dots < 2_1 < 2_2 < \dots < 3_1 < \cdots.$$
Based on this order, we define semi-standard Young tableaux over $\tilde{\N}$, and Knuth equivalence on  words over $\tilde{\N}$. We denote the set of these tableaux by $\SSYT_{\tilde{\N}}$ and Knuth equivalence by $\equiv_{\tilde{\N}}$.

Analogous to the theorems for $\SSYT$, we have the following theorems for $\SSYT_{\tilde{\N}}$.

\begin{thms}[$\tilde{\N}$ analogue of Theorem \ref{inv}]\label{inv_tilde}
Every word over $\tilde{\N}$ is Knuth equivalent to the row-word of a unique semi-standard Young tableau over $\tilde{\N}$ of a normal shape.
\end{thms}

\begin{thms}[$\tilde{\N}$ analogue of Theorem \ref{rect_word}]\label{rect_word_tilde}
Suppose a slide can be performed on $\tilde{L} \in \SSYT_{\tilde{\N}}$ to produce $\tilde{M} \in \SSYT_{\tilde{\N}}$. Then the row-word 
of $\tilde{L}$ is Knuth equivalent to the row-word of $\tilde{M}$, i.e., 
$$w_{row}(\tilde{L}) \equiv_{\tilde{\N}} w_{row}(\tilde{M}).$$
\end{thms}
We can prove Theorems \ref{inv_tilde} and \ref{rect_word_tilde} by translating Theorems \ref{inv} and \ref{rect_word} via some order-preserving maps from $\N$ to $\tilde{\N}$.
Then, we define the operation of jeu de taquin on $\SSYT_{\tilde{\N}}$ with the order in $\tilde{\N}$, and denote the resulting tableau of jeu de taquin on $\tilde{L} \in \SSYT_{\tilde{\N}}$ by $\Rect(\tilde{L})$.

We define the map $\phi_T$ from $\SSYT$ to $\SSYT_{\tilde{\N}}$. For any $L \in \SSYT$ and each $k\in \N$, we label the entries $k$ in $L$ with subscripts $1, 2, \dots$ in increasing order, starting from the bottom row to the top, and from left to right within each row. 
It is obvious that for any $L \in \SSYT$, we can recover $L$ by applying the labeling-off operation to $\phi_T(L)$. Thus, the restriction of $\phi_T$ to its image is invertible, and its inverse, denoted by $\phi_T^{-1}$, is the labeling-off map. 

In parallel, we define the map $\phi_w$ from the set of all words over $\N$ to those over $\tilde{\N}$. For any word $w$ over $\N$ and each $k\in \N$, we label the letters $k$ in $w$ with subscripts $1, 2, \dots$ in increasing order, starting from left to right. 
It is obvious that for any word $w$, we can recover $w$ by applying the labeling-off operation to $\phi_w(w)$. Thus, the restriction of $\phi_w$ to its image is invertible, and its inverse, denoted by $\phi_w^{-1}$, is the labeling-off map.
For any $L \in \SSYT$, we immediately have that $\phi_w (w_{row}(L)) = w_{row}(\phi_T(L))$ by the definitions.\\[-0.5em]

\begin{ex}
For $
L=
~\begin{ytableau}
1 & 1 & 1 \\
2 & 3 & 4 \\
3
\end{ytableau}~$ with $w_{row}(L)=3234111$, we have the following by the definition:
$$\phi_T(L)=
~\begin{ytableau}
1_1 & 1_2 & 1_3 \\
2_1 & 3_2 & 4_1 \\
3_1
\end{ytableau}~, 
\phi_w(w_{row}(L)) = 3_1 2_1 3_2 4_1 1_1 1_2 1_3 = w_{row}(\phi_T(L)).$$
\end{ex}

Then, we have the following proposition.

\begin{prop}\label{rect_phi}
Let $L$ be a semi-standard Young tableau (over $\N$).
Then, we have
$$\Rect(\phi_T(L)) = \phi_T(\Rect(L)).$$
\end{prop}
To prove Proposition \ref{rect_phi}, we show the following lemma.

\begin{lem}\label{lem_phi}
Let $w, w'$ be words over $\N$.
Then, 
$$w \equiv w' \Longrightarrow \phi_w(w) \equiv_{\tilde{\N}} \phi_w(w').$$
\end{lem}

\begin{proof}[Proof of Lemma \ref{lem_phi}]
It suffices to check only the case in which $w$ and $w'$ are related by a single Knuth relation.
\\

\noindent
\textbf{\underline{Case of (K).}} Suppose $w$ and $w'$ are related by (K). Then we can assume that they are of the form $w = u \cdot bca \cdot v$ and $w' = u \cdot bac \cdot v$ for some words $u, v$ and $a, b, c \in \N$ with $a < b \le c$. Let $\alpha, \beta, \gamma$ (resp. $\alpha', \beta', \gamma'$) be the subscripts labeled with the entries originated from $a, b, c$ respectively in $\phi_w(w)$ (resp. $\phi_w(w')$).
Let $\tilde{u}, \tilde{v}$ (resp. $\tilde{u}', \tilde{v}'$) be the words over $\tilde{\N}$ originated from $u, v$ in $\phi_w(w)$ (resp. $\phi_w(w')$).
It means that $\phi_w(w) = \tilde{u} \cdot b_{\beta} c_{\gamma} a_{\alpha} \cdot \tilde{v}$ and $\phi_w(w') = \tilde{u}' \cdot b_{\beta'} a_{\alpha'} c_{\gamma'} \cdot \tilde{v}'$.
Now, we prove that $\phi_w(w') = \tilde{u} \cdot b_{\beta} a_{\alpha} c_{\gamma} \cdot \tilde{v}$.

By the definition of $\phi_w$, we have that $\tilde{u} = \tilde{u}' = \phi_w(u)$. Moreover, the subscripts in $\tilde{v}$ (resp. $\tilde{v}'$) are determined only by the number of each integer in $u \cdot bca$ (resp. $u \cdot bac$). When we count the numbers, there is no difference between the two cases: in $u \cdot bca$ and in $u \cdot bac$. So we have $\tilde{v} = \tilde{v}'$.

The subscript $\alpha$ (resp. $\alpha'$) is the number of the integer $a$ in $u\cdot bc$ (resp. $u \cdot b$) plus $1$. By the condition $a < b \le c$ especially $a \ne c$, we have that $\alpha = \alpha'$.
Both of the subscripts $\beta$ and $\beta'$ are the number of the integer $b$ in $u$ plus $1$. Thus, $\beta = \beta'$. 
The subscript $\gamma$ (resp. $\gamma'$) is the number of the integer $c$ in $u\cdot b$ (resp. $u \cdot ba$) plus $1$. By the condition $a \ne c$, we have that $\gamma = \gamma'$.

Then, we have $\phi_w(w') = \tilde{u} \cdot b_{\beta} a_{\alpha} c_{\gamma} \cdot \tilde{v}$.
If $a_\alpha < b_\beta < c_\gamma$, this means that $\phi_w(w')$ is related to $\phi_w(w)$ by (K). Therefore, we now prove $a_\alpha < b_\beta < c_\gamma$. 
If $b < c$, we have $a_\alpha < b_\beta < c_\gamma$ immediately.
If $b = c$, we have $\beta < \gamma$ because $b$ is left to $c$ in $w$. So we have $a_\alpha < b_\beta < c_\gamma$ and then we have the conclusion that $\phi_w(w')$ is related to $\phi_w(w)$ by (K).
\\

\noindent
\textbf{\underline{Case of (K').}} Suppose $w$ and $w'$ are related by (K'). Then we can assume that they are of the form $w = u \cdot acb \cdot v$ and $w' = u \cdot cab \cdot v$ for some words $u, v$ and $a, b, c \in \N$ with $a \le b < c$. Let $\alpha, \beta, \gamma$ (resp. $\alpha', \beta', \gamma'$) be the subscripts labeled with the entries originated from $a, b, c$ respectively in $\phi_w(w)$ (resp. $\phi_w(w')$).
Let $\tilde{u}, \tilde{v}$ (resp. $\tilde{u}', \tilde{v}'$) be the words over $\tilde{\N}$ originated from $u, v$ in $\phi_w(w)$ (resp. $\phi_w(w')$).
It means that $\phi_w(w) = \tilde{u} \cdot a_{\alpha} c_{\gamma} b_{\beta} \cdot \tilde{v}$ and $\phi_w(w') = \tilde{u}' \cdot c_{\gamma'} a_{\alpha'} b_{\beta'} \cdot \tilde{v}'$.
We can obtain that $\phi_w(w') = \tilde{u} \cdot c_{\gamma} a_{\alpha} b_{\beta} \cdot \tilde{v}$ and $a_\alpha < b_\beta \le c_\gamma$ similarly to the case of (K).
Then we have the conclusion that  $\phi_w(w')$ is related to $\phi_w(w)$ by (K'). 
\end{proof}

\begin{proof}[Proof of Proposition \ref{rect_phi}]

By Theorem \ref{rect_word}, Theorem \ref{rect_word_tilde} and Lemma \ref{lem_phi}, it follows that 
\begin{align*}
w_{row}(\Rect(\phi_T(L)))
&\equiv_{\tilde{\N}} w_{row}(\phi_T(L)) &(\text{by Thm.\ref{rect_word_tilde}})\\
&= \phi_w(w_{row}(L)) & \\
&\equiv_{\tilde{\N}} \phi_w(w_{row}(\Rect(L))) & (\text{by Thm.\ref{rect_word} and Lem.\ref{lem_phi}})\\
&= w_{row}(\phi_T(\Rect(L))).
\end{align*}
By Theorem \ref{inv}, 
\begin{align*}
\Rect(\phi_T(L))
 = \phi_T(\Rect(L)).
\end{align*}
Thus, we obtain Proposition \ref{rect_phi}.
\end{proof}

We define {\it arm} and {\it body} for skew partitions and skew tableaux.

\begin{dfn}[Arm and body]\label{arm_body_def}
Let $\lambda/\mu$ be a skew partition. 
\begin{itemize}
\item The {\it right-arm} of diagram $\lambda/\mu$ is the part of the diagram which consists of the boxes $(1, m + \mu_1), \dots, (1, \lambda_1)$ with $m = {\rm min}\{\lambda_2, \sum_{i \ge 2}(\lambda_i - \mu_i)\}$. We denote the right-arm of $\lambda/\mu$ by $A_r(\lambda/\mu)$. 
\item
The {\it left-arm} of $\lambda/\mu$ is the part of the diagram which consists of the boxes $(n + \mu'_1, 1), \dots, (\lambda'_1, 1)$ with $n = {\rm min} \{\lambda'_2, \sum_{j \ge 2}(\lambda'_j - \mu'_j)\}$. We denote the left-arm of $\lambda/\mu$ by $A_l(\lambda/\mu)$. 
\item
The {\it body} of $\lambda/\mu$ is the part of the diagram obtained by removing $A_r(\lambda/\mu)$ and $A_l(\lambda/\mu)$ from the diagram $\lambda/\mu$, and is denoted by $B(\lambda/\mu)$.
\end{itemize}
Here, $\sum_{i \ge 2}(\lambda_i - \mu_i)$ is the number of boxes in the second row or below in the diagram of $\lambda/\mu$. 
\end{dfn}

\begin{ex}
For $\lambda = (6, 3, 3, 1, 1)$ and $\mu = (2, 1)$, $m = {\rm min}\{\lambda_2, \sum_{i \ge 2}(\lambda_i - \mu_i)\} = {\rm min}\{3, 7\} = 3$ and then $A_r(\lambda/\mu)$ consists of the boxes $\{(1, m + \mu_1), \dots, (1, \lambda_1)\} = \{(1, 5), (1, 6)\}$. Also, $n = {\rm min} \{\lambda'_2, \sum_{j \ge 2}(\lambda'_j - \mu'_j)\} = {\rm min}\{3, 8\} = 3$ and then $A_l(\lambda/\mu)$ consists of $\{(n + \mu'_1, 1), \dots, (\lambda'_1, 1)\} = \{(5, 1)\}$.    
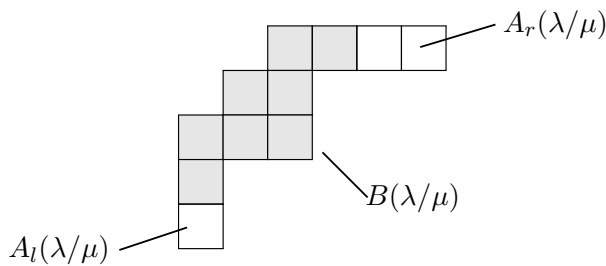
$$
\vcenter{\hbox{
\begin{tikzpicture}[baseline=(current bounding box.center)]
\node at (0,0){ 
~\begin{ytableau}
\none & \none & *(black!10){} & *(black!10){} & {} & {} \\
\none & *(black!10){} & *(black!10){}  \\
*(black!10){} & *(black!10){} & *(black!10){} \\
*(black!10){} \\
{}
\end{ytableau}~};
\draw[line width=0.7pt] (-1.6, -1.2) -- (-2.5, -1.5) node[at={(-3.3, -1.5)}]{$A_l(\lambda/\mu)$};
\draw[line width=0.7pt] (1.5, 1.2) -- (2.5, 1.5) node[at={(3.3, 1.5)}]{$A_r(\lambda/\mu)$};
\draw[line width=0.7pt] (0.2, -0.2) -- (0.8, -0.8) node[at={(1.4, -0.8)}]{$B(\lambda/\mu)$};
\end{tikzpicture}. }}
$$

\end{ex}

For any set $X$ and any tableau $T \in T(\lambda/\mu, X)$, we define $A_r(T)$ as the part of $T$ consisting of the boxes in $A_r(\lambda/\mu)$ and entries in them. We call this the right-arm of $T$. We also define $A_l(T)$ and $B(T)$ similarly. Note that if $m + \mu_1> \lambda_1$ (resp. $n + \mu'_1 > \lambda'_1$), $A_r(\lambda/\mu)$ (resp. $A_l(\lambda/\mu)$) is empty.

\begin{lem}\label{skew_box_lem}
Let $L \in \SSYT(\lambda/\mu)$, $\tilde{L}=(\tilde{l}_{ij}) \coloneq \phi_T(L)$, and $\nu = \shape(\Rect(\tilde{L}))$. Then, 
\begin{enumerate}[rightmargin=0em]
\item $\Rect(\tilde{L})$ has $A_r(\tilde{L})$ in the right-most continuous boxes of the first row. That is, with $m = {\rm min}\{\lambda_2, \sum_{i \ge 2}(\lambda_i - \mu_i)\}$, the entries $\tilde{l}_{1, m + \mu_1}, \dots, \tilde{l}_{1, \lambda_1}$ are placed in the boxes $(1, m + \mu_1 - \lambda_1 + \nu_1), \dots, (1, \nu_1)$ in $\Rect(\tilde{L})$, respectively.
\item $\Rect(\tilde{L})$ has $A_l(\tilde{L})$ in the bottom continuous boxes of the first column. That is, with $n = {\rm min} \{\lambda'_2, \sum_{j \ge 2}(\lambda'_j - \mu'_j)\}$, the entries $\tilde{l}_{n + \mu'_1, 1}, \dots,  \tilde{l}_{\lambda'_1, 1}$ are placed in the boxes $(n + \mu'_1 - \lambda'_1  + \nu'_1, 1), \dots, (\nu'_1, 1)$ in $\Rect(\tilde{L})$, respectively. 
\end{enumerate}
\end{lem}

\begin{ex}
$$
\vcenter{\hbox{
\begin{tikzpicture}[baseline=(current bounding box.center)]
\node at (0,0){ 
~\begin{ytableau}
\none & \none & *(black!10)1_1 & *(black!10)2_4 & 2_5 & 3_3 \\
\none & *(black!10)2_2 & *(black!10)2_3  \\
*(black!10)2_1 & *(black!10)3_2 & *(black!10)4_2 \\
*(black!10)3_1 \\
4_1
\end{ytableau}~};
\draw[line width=0.7pt] (-1.2, -1.2) -- (-0.8, -1.2) node[at={(-0.2, -1.2)}]{$A_l(\tilde{L})$};
\draw[line width=0.7pt] (1, 0.7) -- (1.5, 0.2) node[at={(2, 0.2)}]{$A_r(\tilde{L})$};
\draw[line width=0.7pt] (0.2, -0.2) -- (0.8, -0.8) node[at={(1.4, -0.8)}]{$B(\tilde{L})$};
\end{tikzpicture} }}
\stackrel{\text{jeu de taquin}}{\mapsto}
\vcenter{\hbox{
\begin{tikzpicture}[baseline=(current bounding box.center)]
\node at (0,0){ 
~\begin{ytableau}
*(black!10)1_1 & *(black!10)2_2& *(black!10)2_3 & *(black!10)2_4 & 2_5 & 3_3 \\
*(black!10)2_1 & *(black!10)3_2 & *(black!10)4_2\\
*(black!10)3_1 \\
4_1
\end{ytableau}~};
\draw[line width=0.7pt] (-1.2, -0.9) -- (-0.8, -0.9) node[at={(-0.2, -0.9)}]{$A_l(\tilde{L})$};
\draw[line width=0.7pt] (1, 0.7) -- (1.3, 0.1) node[at={(1.8, 0.1)}]{$A_r(\tilde{L})$};
\end{tikzpicture}. }}$$

\end{ex}

\begin{proof}[Proof of Lemma \ref{skew_box_lem}]
Consider the process of jeu de taquin which is carried out starting from the slide of the right-most entry of the bottom row of the smaller (deleted) diagram in each step.
$$
~
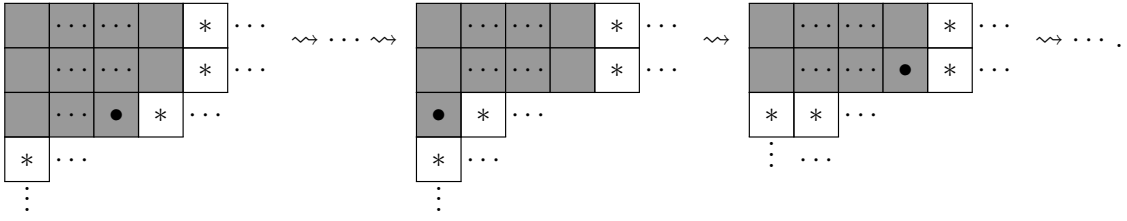
\begin{ytableau}
*(black!40){} & *(black!40)\cdots & *(black!40)\cdots & *(black!40){} & {\ast} & \none[\cdots]\\
*(black!40){} & *(black!40)\cdots & *(black!40)\cdots & *(black!40){} & {\ast} & \none[\cdots]\\
*(black!40){} & *(black!40)\cdots & *(black!40)\bullet & {\ast} & \none[\cdots] \\
{\ast} & \none[\cdots] \\
\none[\vdots] 
\end{ytableau}~
\leadsto
\cdots 
\leadsto
~\begin{ytableau}
*(black!40){} & *(black!40)\cdots & *(black!40)\cdots & *(black!40){} & {\ast} & \none[\cdots]\\
*(black!40){} & *(black!40)\cdots & *(black!40)\cdots & *(black!40){} & {\ast} & \none[\cdots]\\
*(black!40)\bullet & {\ast} & \none[\cdots] \\
{\ast} & \none[\cdots] \\
\none[\vdots] 
\end{ytableau}~
\leadsto 
~\begin{ytableau}
*(black!40){} & *(black!40)\cdots & *(black!40)\cdots & *(black!40){} & {\ast} & \none[\cdots]\\
*(black!40){} & *(black!40)\cdots & *(black!40)\cdots & *(black!40)\bullet & {\ast} & \none[\cdots]\\
{\ast} & {\ast} & \none[\cdots] \\
\none[\vdots] & \none[\cdots]
\end{ytableau}~
\leadsto
\cdots.
$$
Now, suppose that we complete the slides from the second row and below. Let $\tilde{L}'$ be the resulting tableau after these slides have been completed. By the definition of slide, the first row of the tableau is not changed throughout this process. Therefore, the shape of $\tilde{L}'$ is $\psi/(\mu_1)$ for some normal shape $\psi = (\psi_1, \dots, \psi_r)$ and the shape $(\mu_1)$ which consists of one row.
$$
\tilde{L'} = 
\ytableausetup{boxsize=1.7em}
\begin{tikzpicture}[baseline=(current bounding box.center)]
\node at (0,0){ 
~\begin{ytableau}
*(black!40){} & *(black!40)\cdots & *(black!40)\cdots & *(black!40){} & \tilde{l}_{1i} & \cdots & \tilde{l}_{1 \lambda_1}\\
\ast & \ast & \ast & \none[\cdots]\\
\ast & \ast & \none[\cdots]\\
\none[\vdots] & \none[\vdots]
\end{ytableau}~
};
\node at (1.6, 0.48) {\small ($i = \mu_1 + 1$)};
\end{tikzpicture} 
\in \SSYT_{\tilde{\N}}(\psi). 
$$
Since the entire process of a slide is contained within the shape $\lambda$, the length of any row in $\tilde{L}'$ cannot exceed the length of the corresponding row in $\lambda$. In particular, for the second row, we must have $\psi_2 \le \lambda_2$.
Furthermore, the $\psi_2$ entries in the second row of $\tilde{L}'$ must originate from the second row or below in the original tableau $\tilde{L}$. The total number of entries in these rows is $\sum_{i \ge 2}(\lambda_i -\mu_i)$. 
Therefore, we have $\psi_2 \le \sum_{i \ge 2}(\lambda_i - \mu_i)$.
Combining $\psi_2 \le \lambda_2$ and $\psi_2 \le \sum_{i \ge 2}(\lambda_i - \mu_i)$, we obtain that $\psi_2 \le {\rm min}\{\lambda_2, \sum_{i \ge 2}(\lambda_i - \mu_i)\}$. 
Let $m = {\rm min}\{\lambda_2, \sum_{i \ge 2}(\lambda_i - \mu_i)\}$.

Now, let us consider the remaining slides of jeu de taquin on $\tilde{L}$, which is equivalent to the process of jeu de taquin on $\tilde{L}'$. In this process, the entries in the first row slide to the left of the first row or remain in their position, and their relative horizontal order is preserved. 
Since the second row of $\tilde{L}'$ has length $\psi_2$, there are no entries in the $(i, j)$-box with $i \ge  2$ and $j \ge \psi_2 + 1$ in $\tilde{L}'$. 
Since an elementary slide moves an entry only into an adjacent box to its left or above it, no entry from the second row or below can move into the $(1, j)$-box in $\tilde{L}'$ with $j \ge \psi_2 + 1$. 
Then, we obtain that the entries $\tilde{l}_{1, \psi_2 + 1 - \nu_1 + \lambda_1}, \dots, \tilde{l}_{1\lambda_1}$ are placed in the boxes $(1, \psi_2 + 1), \dots, (1, \nu_1)$, respectively, where $\nu = \shape(\Rect(\tilde{L}))$.

It follows that $\nu_1 \ge \lambda_1 - \mu_1$, since all entries from the first row of $\tilde{L}$ move into the first row of $\Rect(\tilde{L})$ by jeu de taquin.  
We now consider two cases: $\nu_1 > \lambda_1 - \mu_1$ and $\nu_1 = \lambda_1 - \mu_1$ separately.
If $\nu_1 > \lambda_1 - \mu_1$, then we have the following calculation: 
\begin{align*}
\psi_2 + 1 - \nu_1 + \lambda_1 
&\le \psi_2 + (\nu_1 - \lambda_1 + \mu_1) - \nu_1 + \lambda_1 \\
&= \psi_2 + \mu_1 \\
&\le m + \mu_1
\end{align*}
by $1 \le \nu_1 - \lambda_1 + \mu_1$ and $\psi_2 \le m$. Then, we have that the entries $\tilde{l}_{1, m + \mu_1}, \dots, \tilde{l}_{1\lambda_1}$ are placed in the boxes $(1, m + \mu_1 - \lambda_1 + \nu_1), \dots, (1, \nu_1)$ in $\Rect(\tilde{L})$, respectively. 
Next, we suppose that $\nu_1 = \lambda_1 - \mu_1$. In this case, the process of jeu de taquin on $\tilde{L}'$ was sliding the entries only leftward. Therefore, the entries $\tilde{l}_{1\mu_1}, \dots, \tilde{l}_{1\lambda_1}$ are placed in the boxes $(1, 1), \dots, (1, \nu_1)$ in $\Rect(\tilde{L})$, respectively. In particular, we have that the entries $\tilde{l}_{1, m + \mu_1}, \dots, \tilde{l}_{1 \lambda_1}$ are placed in the boxes $(1, m + \mu_1 - \lambda_1 + \nu_1), \dots, (1, \nu_1)$ in $\Rect(\tilde{L})$, respectively. 
We conclude the first statement of Lemma \ref{skew_box_lem}. 

The set of entries in $\tilde{L}$ is pairwise distinct. Therefore, the operation of slides are commutative with the operation of transposition. Operating jeu de taquin on the transposition $\tilde{L}^t$, applying the first statement of this lemma to it, and finally  transposing again, we have the second statement of Lemma \ref{skew_box_lem}.
\end{proof}

\subsection{Proof of Theorem \ref{skew}}\label{proof_skew_subs}

Let $\bm{v} = (v_{ij}) \in T(\lambda/\mu, \C), L \in \SSYT(\lambda/\mu)$. Let $\nu = \shape(\Rect(L))$. We define a new tableau $\bm{v}_L \in T(\nu, \C)$ as follows. 
Let $\tilde{L} = \phi_T(L)$. The definition of $\phi_T$ ensures that the set of all entries in $\tilde{L}$ is pairwise distinct. Since the operation of jeu de taquin preserves the set of all entries in $\tilde{L}$, there is a natural bijection $\rho_L : D(\lambda/\mu) \to D(\nu)$, 
which maps $(i, j) \in D(\lambda/\mu)$ to the box in $\Rect(\tilde{L})$ where $\tilde{l}_{ij}$ move into by jeu de taquin on $\tilde{L}$. This means that $\Rect(\tilde{L}) = (\tilde{l}_{\rho_L^{-1}(ij)})_{(i, j) \in D(\nu)}$ where $\tilde{L} = (\tilde{l}_{ij}).$
Using $\rho_L$, we define $\bm{v}_{L} = (v_{\rho_L^{-1}(ij)})_{(i, j) \in D(\nu)}$. It means that we place the entry  $v_{i_0j_0}$ into the box at $(i_1, j_1)$ in $\bm{v}_{L}$ when $(i_0, j_0)$ in $\tilde{L}$ is mapped to $(i_1, j_1)$ in $\Rect(\tilde{L})$. 
Note that $\bm{v}_L$ has the same shape as $\Rect(L)$ and its entries are complex variables from the set $\{v_{ij}\}$, with each variable appearing exactly once. 

\begin{ex}
For $L =
~\begin{ytableau}
*(black!40){} & 2 \\
1 & 3 \\
2
\end{ytableau}~$ and 
$\bm{v} =
~\begin{ytableau}
*(black!40){} & v_{12} \\
v_{21} & v_{22} \\
v_{31}
\end{ytableau}~$,

\begin{align*}
\phi_T(L) 
=
~\begin{ytableau}
*(black!40){} & 2_2 \\
1_1 & 3_1 \\
2_1
\end{ytableau}~ \stackrel{\text{jeu de taquin}}{\mapsto}
~\begin{ytableau}
1_1 & 2_2 \\
2_1 & 3_1 
\end{ytableau}~
=\Rect(\phi_T(L)).\\
\end{align*}
Then,
\begin{align*}
\bm{v} 
=
~\begin{ytableau}
*(black!40){} & v_{12} \\
v_{21} & v_{22} \\
v_{31}
\end{ytableau}~\qquad \mapsto \qquad
~\begin{ytableau}
v_{21} & v_{12} \\
v_{31} & v_{22} 
\end{ytableau}~
=\bm{v}_L. \qquad
\end{align*}

\end{ex}
We define $\mathcal{G}(\lambda/\mu)$ as the set of all partitions $\nu$ such that $c_{\mu\nu}^\lambda>0$.  
Also we define $U_\nu(\bm{v}) = \{\bm{v}_L\>|\>\shape(\Rect(L))=\nu\}.$ 
By Lemma \ref{skew_box_lem} and the definition of $\bm{v}_L$, we have the following corollary.

\begin{cor}\label{skew_box_cor}
Let $L \in \SSYT(\lambda/\mu)$ and $\bm{v} = (v_{ij}) \in T(\lambda/\mu)$. Then, 
\begin{enumerate}[rightmargin=0em]
\item $\bm{v}_L$ has $A_r(\bm{v})$ in the right-most continuous boxes of the first row, i.e. with $m = {\rm min}\{\lambda_2, \sum_{i \ge 2}(\lambda_i - \mu_i)\}$, the entries $v_{1, m + \mu_1}, \dots, v_{1\lambda_1}$ are placed in the boxes $(1, m + \mu_1 - \lambda_1 + \nu_1), \dots, (1, \nu_1)$ in $\bm{v}_L$, respectively.
\item $\bm{v}_L$ has $A_l(\bm{v})$ in the bottom continuous boxes of the first column, i.e. with $n = {\rm min} \{\lambda'_2, \sum_{j \ge 2}(\lambda'_j - \mu'_j)\}$, the entries $v_{n + \mu'_1, 1}, \dots,  v_{\lambda'_1 1}$ are placed in the boxes $(n + \mu'_1 - \lambda'_1  + \nu'_1, 1), \dots, (\nu'_1, 1)$ in $\bm{v}_L$, respectively. 
\end{enumerate}
\end{cor}

\begin{proof}[Proof of Theorem \ref{skew}]
Let $\bm{u}_{\nu}(\bm{v})\in U_\nu(\bm{v})$ for $\nu \in\mathcal{G}(\lambda/\mu)$.
For $L \in \SSYT(\lambda/\mu)$ such that shape of $\Rect(L))$ is $\nu$, we have
\begin{align*}
\frac{1}{L^{\bm{v}}}
&=\frac{1}{(\phi_T^{-1}(\phi_T (L)))^{\bm{v}} }\\
&=\frac{1}{(\phi_T^{-1}{(\Rect (\phi_T (L))))}^{\bm{v}_L} }\\
&=\frac{1}{{\Rect(L)}^{\bm{v}_{L}} }
\end{align*}
where the second equality is obtained by the definition of $\bm{v}_{L}$ and the third is by  Proposition \ref{rect_phi} : $\Rect(\phi_T(L)) = \phi_T(\Rect(L)).$

From Corollary \ref{skew_box_cor}, 
any $\bm{v}_L \in U_\nu(\bm{v})$ has $A_r(\bm{v})$ in the right-most contiguous boxes in the first row, and also has $A_l(\bm{v})$ in the bottom contiguous boxes in the first column. The entries in $B(\bm{v})$ appere exactly once in the parts obtained by removing $A_r(\bm{v})$ and $A_l(\bm{v})$ from $\bm{v}_L$. 
Therefore, $\bm{v}_L$ and $\bm{u}_\nu(\bm{v})$ can be mapped to each other by permuting variables in $B(\bm{v})$. 
Thus, we have
$$\sum_{\Sym(B(\bm{v}))} \frac{1}{{\Rect(L)}^{\bm{v}_L} }
=\sum_{\Sym(B(\bm{v}))}\frac{1}{{\Rect(L)}^{{\bm{u}}_\nu(\bm{v})}}$$
for any $L \in \SSYT(\lambda/\mu)$ with $\shape(\Rect(L))=\nu$.
Then, 
\begin{align*}
\sum_{\Sym(B(\bm{v}))}\zeta_{\lambda/\mu}(\bm{v})
&=\sum_{\Sym(B(\bm{v}))}\sum_{L \in \SSYT(\lambda/\mu)}\frac{1}{L^{\bm{v}}}\\
&=\sum_{\Sym(B(\bm{v}))}\sum_{L \in \SSYT(\lambda/\mu)}\frac{1}{{\Rect(\it L)}^{\bm{v}_L}}\\
&=\sum_{\Sym(B(\bm{v}))}\sum_{L \in \SSYT(\lambda/\mu)}\frac{1}{{\Rect(L)}^{\bm{u}_{\shape(\Rect(L))}(\bm{v})}}\\
&=\sum_{\Sym(B(\bm{v}))}\sum_{\nu \in \mathcal{G}(\lambda/\mu)}\sum_{M \in \SSYT(\nu)}c_{\mu\nu}^\lambda\frac{1}{M^{\bm{u}_\nu(\bm{v})}}\\
&=\sum_{\Sym(B(\bm{v}))}\sum_{\nu \in \mathcal{G}(\lambda/\mu)}c_{\mu\nu}^\lambda \zeta_\nu(\bm{u}_\nu(\bm{v}))
\end{align*}
where the fourth equality is from the Theorem \ref{LRrule2}.
It concludes Theorem \ref{skew}.
\end{proof}

\begin{ex}
For $\lambda=(7,3,1,1), \mu=(2,1)$ and $\bm{v} \in T(\lambda/\mu, \C)$, the following equality is derived from Theorem \ref{skew}. For simplicity, we replace all variables in $B(\bm{v})$ with the same variables $v_0$ and replace the double subscripts of other variables with single subscripts.
\ytableausetup{boxsize=1.4em}
\begin{align*}
\zeta_{\lambda/\mu} &
\left(\vcenter{\hbox{
~\begin{ytableau}
\none & \none & v_0 & v_0 & v_1 & v_2 & v_3 \\
\none & v_0 & v_0 \\
v_0 \\
v_4
\end{ytableau}~
}}\right)\\
& = \zeta_{(7,2)}
\left(\vcenter{\hbox{
~\begin{ytableau}
v_0 & v_0 & v_0 & v_0 & v_1 & v_2 & v_3 \\
v_4 & v_0
\end{ytableau}~
}}\right)
+ \zeta_{(7,1,1)}
\left(\vcenter{\hbox{
~\begin{ytableau}
v_0 & v_0 & v_0 & v_0 & v_1 & v_2 & v_3 \\
v_0 \\
v_4 
\end{ytableau}~
}}\right)\\
& + \zeta_{(6,3)}
\left(\vcenter{\hbox{
~\begin{ytableau}
v_0 & v_0 & v_0 & v_1 & v_2 & v_3 \\
v_4 & v_0 & v_0 
\end{ytableau}~
}}\right)
+ 2\zeta_{(6,2,1)}
\left(\vcenter{\hbox{
~\begin{ytableau}
v_0 & v_0 & v_0 & v_1 & v_2 & v_3 \\
v_0 & v_0 \\
v_4
\end{ytableau}~
}}\right)\\
&+ \zeta_{(6,1,1,1)}
\left(\vcenter{\hbox{
~\begin{ytableau}
v_0 & v_0 & v_0 & v_1 & v_2 & v_3 \\
v_0 \\
v_0 \\
v_4
\end{ytableau}~
}}\right)
+ \zeta_{(5,3,1)}
\left(\vcenter{\hbox{
~\begin{ytableau}
v_0 & v_0 & v_1 & v_2 & v_3 \\
v_0 & v_0 & v_0 \\
v_4
\end{ytableau}~
}}\right)\\
& + \zeta_{(5,2,1,1)}
\left(\vcenter{\hbox{
~\begin{ytableau}
v_0 & v_0 & v_1 & v_2 & v_3 \\
v_0 & v_0 \\
v_0 \\
v_4
\end{ytableau}~
}}\right)
\end{align*}
by 
$c_{\mu(7,2)}^\lambda 
= c_{\mu(7,1,1)}^\lambda 
= c_{\mu(6,3)}^\lambda 
= c_{\mu(6,1,1,1)}^\lambda 
= c_{\mu(5,3,1)}^\lambda 
= c_{\mu(5,2,1,1)}^\lambda = 1$, 
$c_{\mu(6,2,1)}^\lambda = 2$ and $c_{\mu\nu}^\lambda = 0$ for the other partitions $\nu$.

\end{ex}

\subsection{Proof of Theorem \ref{main}}

To prove Theorem \ref{main}, we define the skew shape $\mu \ast \nu$ for partitions $\mu = (\mu_1, \dots, \mu_r), \nu = (\nu_1, \dots, \nu_s)$ with $r = l(\mu), s = l(\nu)$. 
We define $\mu \ast \nu$  by taking a rectangle of empty squares with $\mu_1$ columns and $\nu'_1$ rows (for the smaller diagram), and placing diagram $\mu$ below and $\nu$ to the right of this rectangle.

For a set $X$, let $S \in T(\mu, X), T \in T(\nu, X)$.
We define $S \ast T \in T(\mu \ast \nu, X)$ by placing $S$ below and $T$ to the right of the rectangle. 
For any $M_0 \in \SSYT(\mu)$ and $M_1 \in \SSYT(\nu)$, it is obvious that $M_0 \ast M_1$ is a semi-standard Young tableau.

\begin{ex}
For $\mu = (2, 1), \nu = (3, 2)$, 
$M_0 = 
~\begin{ytableau}
1 & 2 \\
3
\end{ytableau}~$, 
$M_1 = 
~\begin{ytableau}
1 & 2 & 4\\
2 & 3
\end{ytableau}~$, 

\begin{align*}
\mu \ast \nu 
= ~\begin{ytableau}
*(black!40){} & *(black!40){} & {} & {} & {} \\
*(black!40){} & *(black!40){} & {} & {} \\
{} & {} \\
{}
\end{ytableau}~,\  
M_0 \ast M_1 
= ~\begin{ytableau}
*(black!40){} & *(black!40){} & 1 & 2 & 4 \\
*(black!40){} & *(black!40){} & 2 & 3 \\
1 & 2 \\
3
\end{ytableau}~.
\end{align*}

\end{ex}

The following theorem is another combinatorial rule for computing the Littlewood-Richardson coefficients $c_{\mu\nu}^\lambda$.

\begin{thms}[cf. {\cite[\S5]{F}}]\label{LRrule}
Let $\lambda, \mu, \nu$ be partitions.
For any $L \in \SSYT(\lambda)$, 
$$\#\{M \in \SSYT(\mu \ast \nu) \ |\ {\rm Rect}(M) = L\} = c_{\mu\nu}^\lambda.$$
\end{thms}
By Theorem \ref{LRrule2}, Theorem \ref{LRrule} and the definition of $\mathcal{G}$, $U_\lambda(\bm{s} \ast \bm{t}) \ne \emptyset$ for any $\lambda \in \mathcal{G}(\mu \ast \nu)$ and $(\bm{s}, \bm{t}) = ((s_{ij}), (t_{ij})) \in T(\mu,\mathbb{C}) \times T(\nu,\mathbb{C})$. 
We prove Theorem \ref{main} as the special case of Theorem \ref{skew}.

\begin{proof}[Proof of Theorem \ref{main}]

We denote $(v_{ij}) = \bm{s} \ast \bm{t}$. It means that we denote the entry in $(i, j)$ box of $\bm{s} \ast \bm{t}$ by $v_{ij}$ for each $(i, j) \in D(\mu \ast \nu)$. By the definition of $\bm{s} \ast \bm{t}$, $v_{ij} = s_{i - \nu'_1\ j}$ \ for $i \ge \nu'_1 + 1$ and $v_{ij} = t_{i\ j-\mu_1}$ for $j \ge \mu_1 + 1$. 
Let $\psi = (\psi_1, \dots, \psi_{l(\mu)+ l(\nu)})$ be the outer shape of $\mu \ast \nu$ i.e., $\psi_i = \mu_1 + \nu_i$ for $1 \le i \le l(\nu)$ and $\psi_i = \mu_{i - l(\nu)}$ for $i \ge l(\nu) + 1$. We denote the shape of rectangle with $\mu_1$ columns and $\nu'_1$ rows by $(\{\mu_1\}^{\nu'_1})$. 
By applying the first main theorem to the shape $\mu \ast \nu = \psi/(\{\mu_1\}^{\nu'_1})$, we have the following for any $(\bm{u}_{\lambda} \in U_{\lambda}(\bm{s} \ast \bm{t}))_{\lambda \in \mathcal{G}(\mu \ast \nu)}$: 

\begin{align*}
\sum_{\Sym(B(\bm{s} \ast \bm{t}))} \zeta_{\mu \ast \nu}(\bm{s} \ast \bm{t})
=\sum_{\Sym(B(\bm{s} \ast \bm{t}))} \sum_{\lambda \in \mathcal{G}(\mu \ast \nu)} c_{(\{\mu_1\}^{\nu'_1}) \lambda}^{\psi}  \zeta_{\lambda}(\bm{u}_{\lambda}(\bm{s} \ast \bm{t})).
\end{align*}
By the definition of $B(\bm{v})$ for $\bm{v} = (v_{ij})$, $\Sym(B(\bm{s} \ast \bm{t}))$ is the symmetric group which acts on functions in variables $\{v_{ij}\}$, by permuting the variables except 
$v_{1j}$ with ${\rm min}\{\lambda_2, \sum_{i \ge 2}(\lambda_i - \mu_i)\} + \mu_1 \le j \le \lambda_1$ and 
$v_{i1}$ with ${\rm min} \{\lambda'_2, \sum_{j \ge 2}(\lambda'_j - \mu'_j)\} + \mu'_1 \le i \le \lambda'_1$

We rewrite this in terms of $s_{ij}$ and $t_{ij}$ by $v_{ij} = s_{i - \nu'_1\ j}$ for $i \ge \nu'_1 + 1$ and $v_{ij} = t_{i\ j-\mu_1}$ for $j \ge \mu_1 + 1$. 
Here, it is clear that ${\rm min}\{\psi_2, \sum_{i \ge 2}(\psi_i - \mu_1)\} = \psi_2 = \mu_1 + \nu_2$ because $\psi_2 = (\psi_2 - \mu_1)+ \psi_{\nu'_1 + 1} \le  \displaystyle{\sum_{i \ge 2}(\psi_i - \mu_1)}$. 
Therefore, $\Sym(B(\bm{s} \ast \bm{t}))$ is the symmetric group permuting the variables except  
$$\{s_{i1}\>|\>\mu'_2+\nu'_1\le i\le\mu'_1\}\sqcup\{t_{1j}\>|\>\nu_2+\mu_1\le j\le \nu_1\}.$$
Also we have that $c_{(\{\mu_1\}^{\nu'_1}) \lambda}^{\psi} = c_{\mu\nu}^\lambda$ by Theorem\ref{LRrule}. 
Then, we have
\begin{align*}
\sum_{\Sym(B(\bm{s} \ast \bm{t}))} \zeta_{\mu \ast \nu}(\bm{s} \ast \bm{t})
=\sum_{\Sym(B(\bm{s} \ast \bm{t}))} \sum_{\lambda \in \mathcal{G}(\mu \ast \nu)} c_{\mu\nu}^{\lambda} \zeta_{\lambda}(\bm{u}_{\lambda}(\bm{s} \ast \bm{t})).
\end{align*}
It is clear that  
$$M_0^{\bm{s}}M_1^{\bm{t}} = {M_0 \ast M_1}^{\bm{s} \ast \bm{t}}$$
for $M_0 \in \SSYT(\mu)$ and $M_1 \in \SSYT(\nu)$
Moreover, it is clear that the map $\SSYT(\mu) \times \SSYT(\nu) \ni (M_0, M_1) \mapsto M_0 \ast M_1 \in \SSYT(\mu \ast \nu)$ is a bijection. Then, we have
$$\zeta_{\mu}(\bm{s}) \zeta_{\nu}(\bm{t}) = \zeta_{\mu \ast \nu}(\bm{s} \ast \bm{t}),$$
and then Theorem \ref{main} holds.
\end{proof}

\begin{ex}
\ytableausetup{boxsize=1.4em}
For $\mu = (2, 2, 1, 1), \nu = (5, 2), \bm{s} \in T(\mu, \C)$ and $\bm{t} \in T(\nu, \C)$, the following equality is derived from Theorem \ref{main}.
For simplicity, we replace all variables in $B(\bm{s} \ast \bm{t})$ with the same variables $s_0$ and replace the double subscripts of other variables with single subscripts.
\begin{align*}
\zeta_\mu & \left(\vcenter{\hbox{
~\begin{ytableau}
s_0 & s_0 \\
s_0 & s_0 \\
s_0 \\
s_1 
\end{ytableau}~}}\right)
\zeta_\nu \left(\vcenter{\hbox{
~\begin{ytableau}
s_0 & s_0 & s_0 & t_1 & t_2 \\
s_0 & s_0 
\end{ytableau}~}}\right)
= \zeta_{\mu \ast \nu} \left(\vcenter{\hbox{
~\begin{ytableau}
*(black!40){} & *(black!40){} & s_0 & s_0 & s_0 & t_1 & t_2 \\
*(black!40){} & *(black!40){} & s_0 & s_0 \\
s_0 & s_0 \\
s_0 & s_0 \\
s_0 \\
s_1 
\end{ytableau}~}}\right)\\
&=
\zeta_{(7, 4, 1, 1)} \left(\vcenter{\hbox{
~\begin{ytableau}
s_0 & s_0 & s_0 & s_0 & s_0 & t_1 & t_2 \\
s_0 & s_0 & s_0 & s_0 \\
s_0 \\
s_1
\end{ytableau}~}}\right)
+
\zeta_{(7, 3, 2, 1)} \left(\vcenter{\hbox{
~\begin{ytableau}
s_0 & s_0 & s_0 & s_0 & s_0 & t_1 & t_2 \\
s_0 & s_0 & s_0 \\
s_0 & s_0\\
s_1
\end{ytableau}~}}\right)\\
&+
\zeta_{(6, 4, 2, 1)} \left(\vcenter{\hbox{
~\begin{ytableau}
s_0 & s_0 & s_0 & s_0 & t_1 & t_2 \\
s_0 & s_0 & s_0 & s_0 \\
s_0 & s_0\\
s_1
\end{ytableau}~}}\right)
+
\zeta_{(6, 4, 1, 1, 1)} \left(\vcenter{\hbox{
~\begin{ytableau}
s_0 & s_0 & s_0 & s_0 & t_1 & t_2 \\
s_0 & s_0 & s_0 & s_0 \\
s_0 \\
s_0 \\
s_1
\end{ytableau}~}}\right)\\
&+
\zeta_{(6, 3, 2, 2)} \left(\vcenter{\hbox{
~\begin{ytableau}
s_0 & s_0 & s_0 & s_0 & t_1 & t_2 \\
s_0 & s_0 & s_0 \\
s_0 & s_0 \\
s_1 & s_0 
\end{ytableau}~}}\right)
+
2 \zeta_{(6, 3, 2, 1, 1)} \left(\vcenter{\hbox{
~\begin{ytableau}
s_0 & s_0 & s_0 & s_0 & t_1 & t_2 \\
s_0 & s_0 & s_0 \\
s_0 & s_0 \\
s_0 \\
s_1
\end{ytableau}~}}\right)\\
&+
\zeta_{(6, 3, 1, 1, 1, 1)} \left(\vcenter{\hbox{
~\begin{ytableau}
s_0 & s_0 & s_0 & s_0 & t_1 & t_2 \\
s_0 & s_0 & s_0 \\
s_0 \\
s_0 \\
s_0 \\
s_1
\end{ytableau}~}}\right)
+
\zeta_{(6, 2, 2, 2, 1)} \left(\vcenter{\hbox{
~\begin{ytableau}
s_0 & s_0 & s_0 & s_0 & t_1 & t_2 \\
s_0 & s_0 \\
s_0 & s_0\\
s_0 & s_0\\
s_1
\end{ytableau}~}}\right)\\
&+
\zeta_{(6, 2, 2, 1, 1, 1)} \left(\vcenter{\hbox{
~\begin{ytableau}
s_0 & s_0 & s_0 & s_0 & t_1 & t_2 \\
s_0 & s_0 \\
s_0 & s_0\\
s_0 \\
s_0 \\
s_1
\end{ytableau}~}}\right)
+
\zeta_{(5, 4, 2, 1, 1)} \left(\vcenter{\hbox{
~\begin{ytableau}
s_0 & s_0 & s_0 & t_1 & t_2 \\
s_0 & s_0 & s_0 & s_0 \\
s_0 & s_0\\
s_0 \\
s_1
\end{ytableau}~}}\right)\\
&+
\zeta_{(5, 3, 2, 2, 1)} \left(\vcenter{\hbox{
~\begin{ytableau}
s_0 & s_0 & s_0 & t_1 & t_2 \\
s_0 & s_0 & s_0 \\
s_0 & s_0 \\
s_0 & s_0 \\
s_1
\end{ytableau}~}}\right)
+
\zeta_{(5, 2, 2, 2, 1, 1)} \left(\vcenter{\hbox{
~\begin{ytableau}
s_0 & s_0 & s_0 & t_1 & t_2 \\
s_0 & s_0 \\
s_0 & s_0 \\
s_0 & s_0 \\
s_0 \\
s_1 
\end{ytableau}~}}\right)
\end{align*}
by $c_{\mu\nu}^{(7,4,1,1)}=c_{\mu\nu}^{(7,3,2,1)}=c_{\mu\nu}^{(6,4,2,1)}=c_{\mu\nu}^{(6,4,1,1,1)}=c_{\mu\nu}^{(6,3,2,2)}=c_{\mu\nu}^{(6,3,1,1,1,1)}=c_{\mu\nu}^{(6,2,2,2,1)}=c_{\mu\nu}^{(6,2,2,1,1,1)}=c_{\mu\nu}^{(5,4,2,1,1)}=c_{\mu\nu}^{(5,3,2,2,1)}=c_{\mu\nu}^{(5,2,2,2,1,1)}=1$, $c_{\mu\nu}^{(6,3,2,1,1)}=2$ and $c_{\mu\nu}^\lambda=0$ for the other partitions $\lambda$.
\end{ex}

\section{Generalization of Theorem \ref{skew}}

In this section, we generalize Theorem \ref{skew}.
Let $\alpha = (\alpha_1, \alpha_2, \dots), \beta = (\beta_1, \beta_2, \dots)$ and $\delta = (\delta_1, \delta_2, \dots)$ be normal or skew partitions. When they hold the following conditions, we define a new partition $[\alpha|_{l_0} \delta|_{l_1} \beta]$ via its diagram.\\

\noindent \textbf{Conditions}\\[-7mm]
\begin{enumerate}
\item[{[}W1{]}] $D(\alpha)$ has equal to or more than $l_0$ boxes in the right-most column.
\item[{[}W2{]}] $D(\beta)$ has equal to or more than $l_1$ boxes in the bottom row.
\item[{[}W3{]}] $D(\delta)$ has equal to or more than $l_0$ boxes in the left-most column, and equal to or more than $l_1$ boxes in the top row.
\end{enumerate}

We assume that $\alpha$ and $\beta$ are allowed to be $\emptyset$ and then for any $\delta$ and $l_0 = 0$ (resp. $l_1 = 0$) with $\alpha = \emptyset$ (resp. $\beta = \emptyset$) satisfy this condition.

Then, we construct the diagram of $[\alpha|_{l_0} \delta|_{l_1} \beta]$ as follows: first, we paste the right-edge of the highest contiguous $l_0$ boxes in the right-most column in $D(\alpha)$ to the left-edge of the lowest contiguous $l_0$ boxes in the left-most column of $D(\delta)$. Second, we paste the bottom-edge of the left-most contiguous $l_1$ boxes in the bottom row in $D(\beta)$ to the top-edge of the right-most contiguous $l_1$ boxes in the top-row of $D(\delta)$. 

Let $X$ be a set, and let partitions $\alpha, \beta, \delta$ be satisfying the above conditions. For $P \in T(\alpha, X), Q \in T(\beta, X)$ and $T \in T(\delta, X)$, we define $[P|_{l_0} T|_{l_1} Q] \in T([\alpha|_{l_0} \delta|_{l_1} \beta], X)$ as follows: pasting the right-edge of the highest contiguous $l_0$ boxes in the right-most column in $P$ to the left-edge of the lowest contiguous $l_0$ boxes in the left-most column in $T$, and pasting the bottom-edge of the left-most contiguous $l_1$ boxes in the bottom row in $Q$ to the top-edge of the right-most contiguous $l_1$ boxes in the top row in $T$. 
\begin{ex}
\ytableausetup{boxsize=1.2em}
For $\alpha = (2,2,1), \beta = (4,3,3)/(2,1), \delta = (4,4,2)$,
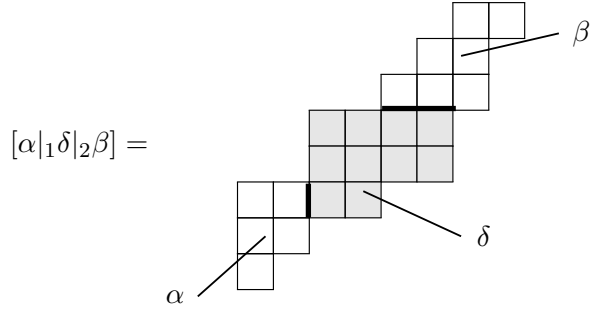
\begin{align*}
[\alpha|_1 \delta|_2 \beta]
= 
\begin{tikzpicture}[baseline=(current bounding box.center)]
\node at (0,0){ 
~\begin{ytableau}
\none & \none & \none & \none & \none & \none & {} & {} \\
\none & \none & \none & \none & \none & {} & {} \\
\none & \none & \none & \none & {} & {} & {} \\
\none & \none & *(black!10){} & *(black!10){} & *(black!10){} & *(black!10){} \\
\none & \none & *(black!10){} & *(black!10){} & *(black!10){} & *(black!10){} \\
{} & {} & *(black!10){} & *(black!10){} \\
{} & {} \\
{}
\end{ytableau}~\ \ };
\draw[line width=2pt] (-0.05,0.5) -- (0.93, 0.5);
\draw[line width=2pt] (-1.03, -0.5) -- (-1.03, -0.95);
\draw[line width=0.7pt] (1, 1.2) -- (2.3, 1.5) node[at={(2.6, 1.5)}]{$\beta$};
\draw[line width=0.7pt] (-1.6, -1.2) -- (-2.5, -2) node[at={(-2.8, -2)}]{$\alpha$};
\draw[line width=0.7pt] (-0.3, -0.6) -- (1, -1.2) node[at={(1.3, -1.2)}]{$\delta$};
\end{tikzpicture}.
\end{align*}

\end{ex}

\begin{rmk}
In \cite{MN}, the tableau $[\alpha |_1 \lambda |_1 \beta]$ is denoted by $[\alpha | \lambda | \beta]$ and called {\it laced tableau} for non-skew or skew partition (or shape) $\alpha, \beta, \delta$. We define the notation $[\alpha |_{l_0} \lambda |_{l_1} \beta]$ inspired by that of laced tableau. 
\end{rmk}

For the case where $\delta$ is a skew partition $\lambda/\mu$, we introduce a new condition. \\

\noindent \textbf{Condition}
\begin{enumerate}
\item[{[}W4{]}] $D(\lambda/\mu)$ has equal to or more than $l_0$ boxes in its left-arm $A_l(\lambda/\mu)$, and equal to or more than $l_1$ boxes in its right-arm $A_r(\lambda/\mu)$.
\end{enumerate}

Then, we have the following theorem. 

\begin{thms}\label{wing}
Let $\lambda/\mu$ be a skew partition, $\alpha, \beta$ be normal or skew partitions, and   $l_0, l_1 \in \N$ such that they satisfy the conditions {\rm [W1]}, {\rm [W2]} and {\rm [W4]}. 
Let $\bm{v} = (v_{ij}) \in T(\lambda/\mu, \C), \bm{a} = (a_{ij}) \in T(\alpha, \C), \bm{b} \in T(\beta, \C)$ be variables. 
Assume that the real parts of the variables 
$v_{1j}$ with ${\rm min}\{\lambda_2, \sum_{i \ge 2}(\lambda_i - \mu_i)\} + \mu_1 \le  j \le \lambda_1 - 1$, 
$v_{i1}$ with ${\rm min} \{\lambda'_2, \sum_{j \ge 2}(\lambda'_j - \mu'_j)\} + \mu'_1 \le i \le \lambda'_1 - 1$, 
and $a_{ij}, b_{ij}$ which are not in the corner of $[\alpha|_{l_0} \lambda/\mu|_{l_1} \beta]$  
are greater than or equal to 1. Assume that the real parts of all except these are greater than 1.
The following equality holds for any $(\bm{u}_\nu(\bm{v}) \in U_{\nu}(\bm{v}))_{\nu \in \mathcal{G}(\lambda/\mu)}$:
\begin{align*}
\sum_{\Sym(B(\bm{v}))} \zeta_{[\alpha|_{l_0} \lambda/\mu|_{l_1} \beta]}([\bm{a}|_{l_0} \bm{v} |_{l_1} \bm{b}])
=\sum_{\Sym(B(\bm{v}))} \sum_{\nu \in \mathcal{G}(\lambda/\mu)} c_{\mu\nu}^{\lambda}\zeta_{\nu}([\bm{a}|_{l_0} \bm{u}_{\nu}(\bm{v}) |_{l_1} \bm{b}]).
\end{align*}
\end{thms}

\begin{rmk}
The condition [W4] is equivalent to $l_0 \le \lambda'_1 - {\rm min} \{\lambda'_2, \sum_{j \ge 2}(\lambda'_j - \mu'_j)\} - \mu'_1 + 1$ because the left-arm of $\bm{v}$ consists of $\lambda'_1 - {\rm min}\{\lambda'_2, \sum_{i \ge 2}(\lambda'_i - \mu'_i)\} - \mu'_1 + 1$ boxes. Similarly, the condition [W4] is equivalent to $l_1 \le \lambda_1 - {\rm min}\{\lambda_2, \sum_{i \ge 2}(\lambda_i - \mu_i)\} - \mu_1 + 1$. \\
\end{rmk}

\begin{proof}
Any semi-standard Young tableau of shape $[\alpha|_{l_0} \lambda/\mu |_{l_1} \beta]$ can be written by the form of $[A|_{l_0} L |_{l_1} B]$ for some $A = (a_{ij}) \in \SSYT(\alpha)$, $B = (b_{ij}) \in \SSYT(\beta)$, $L = (l_{ij}) \in \SSYT(\lambda/\mu)$. 

For $[A|_{l_0} L |_{l_1} B] \in \SSYT([\alpha|_{l_0} \lambda/\mu |_{l_1} \beta])$, we show that $[A|_{l_0} \Rect(L) |_{l_1} B]$ is a semi-standard Young tableau. 
It suffices to check the condition of semi-standard Young tableau on the two edges which paste $A$ with $L$, and $B$ with $L$. 

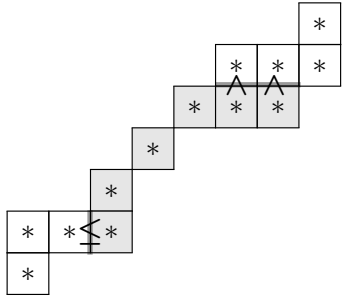
\begin{center}
\begin{tikzpicture}[baseline=(current bounding box.center), line width=2pt]
\ytableausetup{boxsize=1.4em}
\node at (0,0){ 
~\begin{ytableau}
\none & \none & \none & \none & \none & \none & \none & \ast \\
\none & \none & \none & \none & \none & \ast & \ast & \ast \\
\none & \none & \none & \none & *(black!10){\ast} & *(black!10){\ast} & *(black!10){\ast} \\
\none & \none & \none & *(black!10){\ast} \\
\none & \none & *(black!10){\ast} \\
\ast & \ast & *(black!10){\ast} \\
\ast
\end{ytableau}~};
\draw [opacity = 0.4] (0.62, 0.85) --(1.75, 0.85);
\draw [opacity = 0.4] (-1.05, -0.82) --(-1.05, -1.41);
\node[rotate=90] at (0.9, 0.85) {$\bm{>}$};
\node[rotate=90] at (1.4, 0.85) {$\bm{>}$};
\node at (-1.05, -1.1) {$\bm{\le}$};
\end{tikzpicture}

\end{center}

Let $\nu$ be the shape of $\Rect(L)$.
For $1 \le i \le \lambda'_1 - {\rm min}\{\lambda'_2, \sum_{i \ge 2}(\lambda'_i - \mu'_i)\} - \mu'_1 + 1$, each entry in the box $(\nu'_1 - i + 1, 1)$ of $\Rect(L)$ is $l_{\lambda'_1 - i + 1, 1}$. This is because the labeled entry in $\phi_T(L)$ originated from $l_{\lambda'_1 - i + 1, 1}$ is placed in $(\nu'_1 - i + 1, 1)$ in $\Rect(\phi_T(L))$  by Lemma \ref{skew_box_lem}, and we can recover the entry $l_{\lambda'_1 - i + 1\ 1}$ by $\phi_T^{-1}$ from the entry in $\Rect(\phi_T(L)) = \phi_T(\Rect(L))$ by Proposition \ref{rect_phi}. By the condition $l_0 \le \lambda'_1 - {\rm min}\{\lambda'_2, \sum_{i \ge 2}(\lambda'_i - \mu'_i)\} - \mu'_1 + 1$, especially for $1 \le i \le l_0$, each entry in the box $(\nu'_1 - i + 1, 1)$ in $\Rect(L)$ is $l_{\lambda'_1 - i + 1, 1}$. 
Each $l_{\lambda'_1 - i + 1, 1}$ is equal to or larger than $a_{i \alpha'_1}$ because $l_{\lambda'_1 - i +1, 1}$ is adjacent to $a_{i \alpha'_1}$ by the right edge of its box in $[A|_{l_0} L |_{l_1} B] \in \SSYT([\alpha|_{l_0} \lambda/\mu |_{l_1} \beta])$. 
Therefore, the entry in the box $(\nu'_1 - i + 1, 1)$ of ${\rm Rect}(L)$ is larger than $a_{i\alpha'_1}$. It concludes that the condition of semi-standard Young tableau on the vertical edge which pasting $A$ and $\Rect(L)$ is satisfied. 
We can prove that the condition of semi-standard Young tableau on the horizontal edge which past $B$ and $\Rect(L)$ is satisfied in the same way  as the case of the vertical edge which past $A$ and $\Rect(L)$. 
Therefore, we have that $[A|_{l_0} \Rect(L) |_{l_1} B]$ is a semi-standard Young tableau. 

Conversely, for any $[A|_{l_0} M |_{l_1} B] \in \SSYT([\alpha|_{l_0} \nu|_{l_1} \beta])$ and $L = (l_{ij}) \in \SSYT(\lambda/\mu)$ with $\Rect(L) = M$, $[A|_{l_0} L |_{l_1} B]$ is a semi-standard Young tableau. It is because that each $l_{\lambda'_1 - i + 1, 1}$ is equal to or larger than $a_{i \alpha'_1}$ since $l_{\lambda'_1 - i +1, 1}$ is adjacent to $a_{i \alpha'_1}$ by the right edge of its box, and each $l_{1\ \lambda_1 - j + 1}$ is strictly larger than $b_{\beta_1 j}$ since $l_{1, \lambda_1 - j +1}$ is adjacent to $b_{\beta_1 j}$ by the top edge of its box in $[A|_{l_0} M |_{l_1} B]$. Combining this and $[A|_{l_0} M |_{l_1} B] \in \SSYT([\alpha|_{l_0} \nu |_{l_1} \beta])$, we have that $[A|_{l_0} L |_{l_1} B]$ is a semi-standard Young tableau.
Therefore, for the map 
\begin{align*}
{\Rect}_{[\alpha|_{l_0} \lambda/\mu |_{l_1} \beta]} :\  
&\SSYT([\alpha|_{l_0} \lambda/\mu |_{l_1} \beta]) && \to && \bigcup_{\nu \in\mathcal{G}(\lambda/\mu)}\SSYT([\alpha|_{l_0} \nu |_{l_1} \beta])\\
&\hspace{3em}\rotatebox{90}{$\in$}&& &&\hspace{3em}\rotatebox{90}{$\in$}
\\
&[A|_{l_0} L|_{l_1} B]&& \mapsto &&[A|_{l_0} {\rm Rect}(L) |_{l_1} B] ,
\end{align*}
we have that 
$$\#\{W_1 \in \SSYT([\alpha|_{l_0} \lambda/\mu |_{l_1} \beta])\ |\ {\Rect}_{[\alpha|_{l_0} \lambda/\mu |_{l_1} \beta]} (W_1) = W_0  \} = c_{\mu\nu}^\lambda$$
for any $W_0 \in \SSYT([\alpha|_{l_0} \nu |_{l_1} \beta])$ by Theorem \ref{LRrule2}. 
Also it is clear that 
\begin{align*}
[A|_{l_0} L|_{l_1} B]^{[\bm{a}|_{l_0} \bm{v}|_{l_1} \bm{b}]} 
&= A^{\bm{a}} L^{\bm{v}} B^{\bm{b}}\\
&= A^{\bm{a}} \Rect(L)^{\bm{v}_L} B^{\bm{b}}
\end{align*}
by the definition. 
For any $\bm{u}_\nu(\bm{v}) \in U_\nu(\bm{v})$ and $L \in \SSYT(\lambda/\mu)$ such that the shape of $\Rect(L)$ is $\nu$, recall that
$$\sum_{\Sym(B(\bm{v}))} \frac{1}{{\Rect(L)}^{\bm{v}_L} }
=\sum_{\Sym(B(\bm{v}))} \frac{1}{{\Rect(L)}^{{\bm{u}}_\nu(\bm{v})}}.$$
From the discussion above, we have the following calculation, completing the proof.
\begin{align*}
\sum_{\Sym(B(\bm{v}))}& \zeta_{[\alpha|_{l_0} \lambda/\mu|_{l_1} \beta]}([\bm{a}|_{l_0} \bm{v}_L|_{l_1} \bm{b}])\\
&=\sum_{\Sym(B(\bm{v}))}\sum_{W_1 \in \rm{SSYT}([\alpha|_{l_0} \lambda/\mu|_{l_1} \beta])} \frac{1}{W_1^{[\bm{a}|_{l_0} \bm{v}|_{l_1} \bm{b}]}}\\
&=\sum_{\Sym(B(\bm{v}))}\sum_{[A|_{l_0} L|_{l_1} B] \in \SSYT([\alpha|_{l_0} \lambda/\mu|_{l_1} \beta])} \frac{1}{[A|_{l_0} L|_{l_1} B]^{[\bm{a}|_{l_0} \bm{v}|_{l_1} \bm{b}]}}\\
&=\sum_{\Sym(B(\bm{v}))}\sum_{[A|_{l_0} L|_{l_1} B] \in \rm{SSYT}([\alpha|_{l_0} \lambda/\mu|_{l_1} \beta])} \frac{1}{A^{\bm{a}} {\rm Rect}(L)^{\bm{v}_L} B^{\bm{b}}}\\
&=\sum_{\Sym(B(\bm{v}))}\sum_{[A|_{l_0} L|_{l_1} B] \in \SSYT([\alpha|_{l_0} \lambda/\mu|_{l_1} \beta])} \frac{1}{A^{\bm{a}} {\rm Rect}(L)^{\bm{u}_{\shape(\Rect(L))}} B^{\bm{b}}}\\
&=\sum_{\Sym(B(\bm{v}))}\sum_{\nu \in \mathcal{G}(\lambda/\mu)}\sum_{[A|_{l_0} M |_{l_1} B] \in \SSYT([\alpha|_{l_0} \nu |_{l_1} \beta])} c_{\mu\nu}^\lambda \frac{1}{[A|_{l_0} M |_{l_1} B]^{[\bm{a}|_{l_0} \bm{u}_\nu(\bm{v})|_{l_1} \bm{b}]}}\\
&=\sum_{\Sym(B(\bm{v}))}\sum_{\nu \in \mathcal{G}(\lambda/\mu)}\sum_{W_0 \in \SSYT([\alpha|_{l_0} \nu |_{l_1} \beta])} c_{\mu\nu}^\lambda \frac{1}{W_0^{[\bm{a}|_{l_0} \bm{u}_\nu(\bm{v})|_{l_1} \bm{b}]}}\\
&=\sum_{\Sym(B(\bm{v}))}\sum_{\nu \in \mathcal{G}(\lambda/\mu)} c_{\mu\nu}^\lambda \zeta_\nu([\bm{a}|_{l_0} \bm{u}_\nu(\bm{v})|_{l_1} \bm{b}]).
\end{align*}
Thus, we obtain Theorem \ref{wing}.
\end{proof}

\begin{ex}

For $\lambda/\mu =(5,2,1,1)/(2,1)$, $\alpha = (2, 1)$, $\beta = (3, 3)/(2)$, $l_0 = 1, l_1 = 2$, $\bm{v} \in T(\lambda/\mu, \C)$, $\bm{a} \in T(\alpha, \C)$ and $\bm{b} \in T(\beta, \C)$, the following equality is derived from Theorem \ref{wing}. 
For simplicity, we replace all variables in $B(\bm{v})$ with the same variables $v_0$ and replace the double subscripts of other variables $v_{ij}$ with single subscripts.
\ytableausetup{boxsize=1.4em}
\begin{align*}
& \zeta_{[\alpha|_1 \lambda/\mu |_2 \beta]} 
\left(\vcenter{\hbox{
\begin{tikzpicture}[baseline=(current bounding box.center), line width=2pt]
\node at (0,0){ 
~\begin{ytableau}
\none & \none & \none & \none & \none & \none & \none & b_{13}\\
\none & \none & \none & \none & \none & b_{21} & b_{22} & b_{23} \\
\none & \none & \none & \none & *(black!10){v_1} & *(black!10){v_2} & *(black!10){v_3} \\
\none & \none & \none & *(black!10){v_0} \\
\none & \none & *(black!10){v_0} \\
a_{11} & a_{12} & *(black!10){v_4} \\
a_{21}
\end{ytableau}~};
\draw (0.62, 0.85) --(1.77, 0.85);
\draw (-1.05, -0.82) --(-1.05, -1.43);
\end{tikzpicture}
}}\right)\\
& = \zeta_{[\alpha|_1 (5,1)|_2 \beta]}
\left(\vcenter{\hbox{
\begin{tikzpicture}[baseline=(current bounding box.center), line width=2pt]
\node at (0,0){ 
~\begin{ytableau}
\none & \none & \none & \none & \none & \none & \none & b_{13}\\
\none & \none & \none & \none & \none & b_{21} & b_{22} & b_{23} \\
\none & \none & *(black!10){v_0} & *(black!10){v_0} & *(black!10){v_1} & *(black!10){v_2} & *(black!10){v_3} \\
a_{11} & a_{12} & *(black!10){v_4} \\
a_{21}
\end{ytableau}~};
\draw (0.6, 0.3) --(1.76, 0.3);
\draw (-1.05, -0.28) --(-1.05, -0.84);
\end{tikzpicture}
}}\right)
+ \zeta_{[\alpha |_1 (4,2) |_2 \beta]}
\left(\vcenter{\hbox{
\begin{tikzpicture}[baseline=(current bounding box.center), line width=2pt]
\node at (0,0){ 
~\begin{ytableau}
\none & \none & \none & \none & \none & \none & b_{13}\\
\none & \none & \none & \none & b_{21} & b_{22} & b_{23} \\
\none & \none & *(black!10)v_0 & *(black!10)v_1 & *(black!10)v_2 & *(black!10)v_3 \\
a_{11} & a_{12} & *(black!10)v_4 & *(black!10)v_0 \\
a_{21}
\end{ytableau}~};
\draw (0.35, 0.3) --(1.47, 0.3);
\draw (-0.78, -0.28) --(-0.78, -0.86);
\end{tikzpicture}
}}\right)\\
& + 2\ \zeta_{[\alpha |_1 (4,1,1) |_2 \beta]}
\left(\vcenter{\hbox{
\begin{tikzpicture}[baseline=(current bounding box.center), line width=2pt]
\node at (0,0){ 
~\begin{ytableau}
\none & \none & \none & \none & \none & \none & b_{13}\\
\none & \none & \none & \none & b_{21} & b_{22} & b_{23} \\
\none & \none & *(black!10)v_0 & *(black!10)v_1 & *(black!10)v_2 & *(black!10)v_3 \\
\none & \none & *(black!10)v_0 \\
a_{11} & a_{12} & *(black!10)v_4 \\
a_{21}
\end{ytableau}~};
\draw (0.35, 0.55) --(1.5, 0.55);
\draw (-0.78, -0.56) --(-0.78, -1.14);
\end{tikzpicture}
}}\right)
+ \zeta_{[\alpha|_1 (3,2,1) |_2 \beta]}
\left(\vcenter{\hbox{
\begin{tikzpicture}[baseline=(current bounding box.center), line width=2pt]
\node at (0,0){ 
~\begin{ytableau}
\none & \none & \none & \none & \none & b_{13}\\
\none & \none & \none & b_{21} & b_{22} & b_{23} \\
\none & \none & *(black!10)v_1 & *(black!10)v_2 & *(black!10)v_3 \\
\none & \none & *(black!10)v_0 & *(black!10)v_0 \\
a_{11} & a_{12} & *(black!10)v_4 \\
a_{21}
\end{ytableau}~};
\draw (0.08, 0.53) --(1.2, 0.53);
\draw (-0.45, -0.56) --(-0.45, -1.15);
\end{tikzpicture}
}}\right)\\
& + \zeta_{[\alpha|_1 (3,1,1,1) |_2 \beta]}
\left(\vcenter{\hbox{
\begin{tikzpicture}[baseline=(current bounding box.center), line width=2pt]
\node at (0,0){ 
~\begin{ytableau}
\none & \none & \none & \none & \none & b_{13}\\
\none & \none & \none & b_{21} & b_{22} & b_{23} \\
\none & \none & *(black!10)v_1 & *(black!10)v_2 & *(black!10)v_3 \\
\none & \none & *(black!10)v_0 \\
\none & \none & *(black!10)v_0 \\
a_{11} & a_{12} & *(black!10)v_4 \\
a_{21}
\end{ytableau}~};
\draw (0.095, 0.84) --(1.2, 0.84);
\draw (-0.5, -0.85) --(-0.5, -1.42);
\end{tikzpicture}
}}\right).
\end{align*}

\end{ex}

\section*{Acknowledgment}
The author would like to express her sincere gratitude to her supervisor, Professor Yasuo Ohno, for introducing her to multiple zeta values, and giving her his invaluable guidance on both the subject and the attitude towards research. The author appreciates Professor Maki Nakasuji for her numerous helpful comments from the initial stage of this work, and for suggesting the improvement of the main theorem in Section 4. The author is also grateful to Professors Wataru Takeda, and Hideki Murahara for valuable discussions on the research topic in this article, and for valuable comments about this draft. The author also appreciates Professor Shin-ichiro Seki for reading this draft and giving her many valuable comments. 
The author was partly supported by AIE-WISE Program for AI Electronics by Tohoku University.

\vspace{3mm}

\end{document}